\documentclass[12pt]{article}

\usepackage{amsmath}
\usepackage{latexsym}
\usepackage{amsfonts}
\usepackage{amssymb}
\usepackage{amscd}
\usepackage{theorem}

\usepackage[latin1]{inputenc}
\newtheorem{theorem}{Theorem}[section]

\newtheorem{lemma}[theorem]{Lemma}
\newtheorem{proposition}[theorem]{Proposition}
\newtheorem{corollary}[theorem]{Corollary}

{\theorembodyfont{\rmfamily}

}

\newenvironment{proof}{    
  \noindent
  \textbf{Proof.}}{
  \hfill $\Box$
  \vspace{3mm}
}

\numberwithin{equation}{section}

\newcommand{\N}{\mathbb{N}} 
\newcommand{\R}{\mathbb{R}} 
\newcommand{\C}{\mathbb{C}} 
\newcommand{\D}{\mathbb{D}} 

\newcommand{\cL}{{\mathcal L}}

\newcommand{\g}{\gamma}
\newcommand{\wh}{\widehat}

\newcommand{\su}{\subseteq}

\newcommand{\sC}{\mathsf{C}}
\newcommand{\ov}{\overline}

\DeclareMathOperator*{\ind}{ind}    

\begin{document}

\title{Ces\`aro operators on the space of analytic functions with logarithmic
growth}

\author{Jos\'{e} Bonet  }

\date{}

\maketitle

\begin{abstract}
Continuity, compactness, the spectrum and ergodic properties of  Ces\`aro operators are investigated when they act on the space $VH(\D)$ of analytic functions with logarithmic growth on the open unit disc $\D$ of the
complex plane. The space $VH(\D)$ is a countable inductive limit of weighted Banach spaces of analytic functions with compact linking maps. It was introduced and studied by Taskinen and also by Jasiczak.

\end{abstract}

\renewcommand{\thefootnote}{}
\footnotetext{\emph{Article accepted for publication in Annales Polonici Mathematici} }


\section{Introduction.}
The aim of this note is to study Ces\`aro operators when they act on the space $VH(\D)$ of analytic functions with logarithmic growth on the open unit disc $\D$ of the complex plane. This space was introduced and investigated by Taskinen in \cite{T}. It was studied also by Jasiczak \cite{J}. This space is a countable union of weighted Banach spaces of analytic functions whose topology can be described by weighted sup-seminorms.  They proved that the Bergman projection is continuous from the corresponding space of measurable functions into this space. This fact permitted Taskinen and the author to study Toeplitz operators on this space in \cite{BT}.

We denote by $H(\D)$ the Fr\'echet space of all analytic functions on $\D$
endowed with the topology $\tau_{co}$ of uniform convergence on the compact subsets of $\D$.

The classical Ces\`aro operator $\sC$ is given by
\begin{equation}\label{eq.op-C}
f\mapsto \sC(f)\colon z\mapsto  \frac{1}{z}\int_0^z \frac{f(\zeta)}{1-\zeta} d \zeta, \ \ \ z \in \D\setminus\{0\}, \quad {\rm and }\quad \sC(f)(0)=f(0),
\end{equation}
for $f\in H(\D)$. It is a Fr\'echet space isomorphism of $H(\D)$ onto itself, see \cite{Pe}. In terms of the Taylor coefficients
$\wh{f}(n):= \frac{f^{(n)}(0)}{n!}$, for $ n \in \N_0,$ of functions $f(z)= \sum_{n=0}^\infty \wh{f}(n) z^{n} \in H(\D)$ one has the description
$$
\sC (f)(z) = \sum_{n=0}^{\infty} \Big( \frac{1}{n+1} \sum_{k=0}^{n} \wh{f}(k) \Big) z^n, \ \ \ \ z \in \D.
$$

It is known that there are many classical Banach spaces $X$ of analytic functions on $\D$ such that the Ces\`aro operator $\sC$ acts continuously from $X$ into itself; for instance, the Hardy spaces $H^p(\D), 1 \le p<\infty,$ the Bergman  and the  Dirichlet spaces, etc.\ See for example \cite{ABR1,APe,Pe} and the references therein. On the other hand, $\sC $  fails to act in the space $H^\infty (\D)$  of bounded analytic functions, since $\sC(\mathbf{1})(z)= (1/z) \log(1/(1-z))$, for $ z \in \D$.

We describe the space $VH(\D)$  on which the Ces\`aro operators are investigated in this paper.  A {\it weight} $v$ on $\D$  is a continuous function  $v: [0, 1) \to (0,  \infty )$, which is non-increasing on $[0,1)$ and satisfies $\lim_{r \rightarrow 1} v(r)=0$. We extend $v$ to $\D$ radially by $v(z):= v(|z|)$. For such a weight $v$, we  define the following {\it weighted Banach spaces of
analytic functions} on $\D$ $$ H_v^\infty:= \{ f \in H(\D):\ \|f\|_v
:= \sup_{z \in \D} v(z)|f(z)| < \infty \}, $$ $$ H_v^0:= \{ f \in
H(\D):\ v|f| \ {\rm vanishes \ at \ infinity} \ {\rm on }\ \D \},
$$
endowed with the norm $\|.\|_v$. Recall that a  function $g$ vanishes at infinity on $\D$ if for
every $\varepsilon>0$ there is a compact subset $K$ of $\D$ such
that $|g(z)| < \varepsilon$ if $z \notin K$.

For an analytic  function $f \in H(\{z \in \C ; |z| <R \})$ and $r <R$, we denote $M(f,r):= \max\{|f(z)| \ ; \ |z|=r\}$. Using the notation $O$ and $o$ of Landau, $f \in H_v^\infty$ if and only if $M(f,r)=O(1/v(r)), r \rightarrow 1,$ and $f \in H_v^0$ if and only if $M(f,r)=o(1/v(r)), r \rightarrow 1$. Polynomials are contained in $H^0_v$ and  the closure of the polynomials in $H_v^\infty$ coincides with $H_v^0$, see e.g.\ \cite{BBG}.

The examples of weights which are relevant in this paper are:

\noindent $(i)$ $w_{\gamma}(z)=(1-|z|)^{\gamma}, \ z \in \D,$ with $\gamma >0$, which are the so-called standard weights on the disc, for which $A^{-\g} = H^\infty_{w_{\g}}(\D)$ and $A^{-\g}_0 = H^0_{w_{\g}}(\D)$ are the Korenblum type Banach growth spaces; see Section 4.3 in \cite{HKZ}.

\noindent $(ii)$ $v(z)= 1$ if $|z| \leq 1-1/e$, and $v(z)= (-\log(1-|z|))^{-1}$ if $|z| \geq 1-1/e$. Observe that $0 < v(z) \leq v(0)=1$ for all $z \in \D$. The logarithmic weights are defined by $v_{\alpha}(z):= v(z)^\alpha, \alpha \geq 1$.

Banach spaces of the type mentioned above appear naturally in the
study of growth conditions of analytic functions and have been
considered in many papers. We refer to \cite{BBG,BBT,B} and the references therein. Lusky \cite{L2} obtained the isomorphic classification of these spaces. We refer the reader to \cite{HKZ,Z} for unexplained notation.
In what follows $\N$ stands for the natural numbers and  we set $\N_0 := \N \cup \{ 0 \}$.

The space $VH(\D)$ of analytic functions of logarithmic growth introduced by Taskinen in \cite{T} is defined as follows: Set $v_k(z)= v(z)^k, \ z \in \D, k \in \N,$ for $v(z)= 1$ if $|z| \leq 1-1/e$, and $v(z)= (-\log(1-r))^{-1}$ if $|z| \geq 1-1/e$. Since $v_{k+1}(z) \leq v_k(z)$ for all $z \in \D, k \in \N,$ the weighted Banach spaces satisfy $H^\infty_{v_k} \subset H^\infty_{v_{k+1}}$ with continuous inclusion. Then $VH(\D):= \bigcup_{k \in \N} H^\infty_{v_k}$ and it is endowed with the finest locally convex topology such that the inclusions $H^\infty_{v_k} \subset VH(\D)$ are continuous; that is $VH(\D)= {\rm ind}_k H^\infty_{v_k}$ is an (LB)-space, a (Hausdorff) countable inductive limit of the weighted Banach spaces $H^\infty_{v_k}, k \in \N$.

Moreover, $\lim_{|z|\rightarrow 1} \frac{v_{k+1}(z)}{v_{k}(z)}=0$ for each $k \in \N$. This implies that $H^\infty_{v_k} \subset H^0_{v_{k+1}}$ and that the inclusion is in fact compact. Consequently, the space $VH(\D)$ satisfies $VH(\D)=\bigcup_{k \in \N} H^0_{v_k}$, it is a (DFS)-space and its topology can be described by weighted sup-seminorms, see \cite{BMS}, \cite{J}, \cite{T}. In fact in \cite{T} and \cite{J} the space $VH(\D)$ is introduced with a family of weighted sup-seminorms. We refer the reader to \cite{Bi} and \cite{MV} for more information about (DFS)-spaces. Every bounded subset of $VH(\D)$ is relatively compact and it is contained and relatively compact in a step $H^\infty_{v_k}$. In particular, $VH(\D)$ is a Montel space.  More information about this space can be found in \cite{J}, \cite{T}. For example, this space is a topological algebra which is Schwartz but not nuclear.

The Banach space $H^\infty(\D)$ of bounded analytic functions on $\D$ is contained in $VH(\D)$. In fact $H^\infty(\D) \subset H^\infty_{v_1}$ with continuous inclusion. Usually functions $g \in H^\infty_{v_1}$ are called functions with logarithmic mean growth. Every analytic function $g$ in the Bloch space $\mathcal{B}$ has logarithmic mean growth, see e.g.\ page 106 in \cite{Z}. There are analytic functions $g \notin \mathcal{B}$ which have logarithmic mean growth  \cite{GGP}.

On the other hand, for each $\gamma > 0$ and $k \in \N$, there is $C_{\gamma,k} >0$ such that $w_{\gamma}(r)=(1-r)^{\gamma} \leq C_{\gamma,k} v_k(z)$ for each $z \in \D$. Therefore $VH(\D) \subset \bigcap_{\g > 0} A^{-\g} = \bigcap_{\g > 0} A^{-\g}_0$.

The fact that $\sC(\mathbf{1})(z)= (1/z) \log(1/(1-z))$ indicates that the space $VH(\D)$ is a natural frame to investigate the Ces\`aro operator on it. We show in Proposition \ref{continuity} that the Ces\`aro operator $\sC$ is continuous on $VH(\D)$, but it does not act continuously in any of its steps. Its spectrum is exhibited in our main result Theorem \ref{spectrumCesaro}. As a consequence we deduce that $\sC$ is not compact, not power bounded, not mean ergodic and not supercyclic in Corollary \ref{noncompactCesaro} and Proposition \ref{meanergodicCesaro}. The behaviour of the generalized Ces\`aro operator $C_t, \ 0 \leq t < 1,$ on $VH(\D)$ is investigated in section \ref{genCesoper}. In contrast with $\sC$, this operator is compact, power bounded and uniformly mean ergodic.

\section{Ces\`aro operator. }\label{Cesoper}

In this section we investigate the continuity, the spectrum and the ergodic properties of the Ces\`aro operator $\sC$ when it acts on the space $VH(\D)$ studied by  Jasiczak and Taskinen.

The following consequence of Grothendieck's factorization Theorem \cite[Theorem 24.33]{MV} is well-known.

\begin{lemma}\label{continuityLB}
Let $X={\rm ind}_k X_k$ be an (LB)-space. A linear operator $T: X \rightarrow X$ is continuous if and only if for each $k$ there is $m \geq k$ such that $T: X_k \rightarrow X_m$ is continuous.
\end{lemma}

\begin{proposition}\label{continuity}
The Ces\`aro operator $\sC: VH(\D) \rightarrow VH(\D)$ is continuous, but for each $k \in \N$ it does not act continuously on $H^\infty_{v_k}$.
\end{proposition}
\begin{proof}
We begin with the following estimate for $f \in H(\D)$ and $0 < |z| < 1$:
$$
|\sC f(z)| = \big| \frac{1}{z} \int_0^z \frac{f(\zeta)}{1-\zeta} d\zeta \big| = \big| \int_0^1 \frac{f(tz)}{1-tz} dt \big| \leq
$$
$$
\leq M(f,|z|)  \int_0^1 \frac{1}{|1-tz|} dt  \leq M(f,|z|)  \int_0^1 \frac{1}{1-t|z|} dt  =  M(f,|z|) \frac{1}{|z|} \log\Big(\frac{1}{1-|z|}\Big).
$$
This implies, for each $0 < r < 1$,  $$M(\sC f,r) \leq M(f,r) \frac{1}{r} \log\Big(\frac{1}{1-r}\Big).$$
In particular, $\sC$ is continuous on $H(\D)$, as is known \cite{Pe}.

Now we show that $\sC(H^\infty_{v_k}) \subset H^\infty_{v_{k+1}}$ for each $k \in \N$. Fix $k \in \N$ and $f \in H^\infty_{v_k}$. There is $C>0$ such that $|f(z)| \leq C/v_k(z)$ for each $z \in \D$. If $|z| \geq 1-1/e$, then
$$
|\sC f(z)| \leq   M(f,|z|) \frac{1}{|z|} \log\Big(\frac{1}{1-|z|}\Big) \leq  \frac{C}{1-1/e} \frac{1}{v_{k+1}(z)}.
$$
For $|z| = 1-1/e$ we have $|\sC f(z)| \leq C/(1-1/e)$. By the maximum modulus theorem,
$$
\max_{|z| \leq 1-1/e} |\sC f(z)| \leq C/(1-1/e).
$$
Therefore,
$$
\sup_{z \in \D} v_{k+1}(z) |\sC f(z)| \leq C/(1-1/e),
$$
and $\sC f \in H^\infty_{v_{k+1}}$.

The continuity of $\sC$ on $H(\D)$ permits us to apply the closed graph theorem to conclude  that $\sC: H^\infty_{v_k} \rightarrow H^\infty_{v_{k+1}}$ is continuous. Lemma \ref{continuityLB} implies that $\sC: VH(\D) \rightarrow VH(\D)$ is continuous.

It remains to show that, for each $k \in \N$, the operator $\sC$ does not act continuously in each $H^\infty_{v_k}$. To see this, observe that the function $g(z):= \log(1-z), z \in \D,$ is analytic on $\D$ and $|g(z)| \leq -\log(1-|z|)$ for each $z \in \D$. This yields $g^k \in H^\infty_{v_k}$ for each $k \in \N$. Moreover, $$\sC(g^k)(z)= -\frac{1}{(k+1)z}g(z)^{k+1},$$ which is an analytic function which does not belong to $H^\infty_{v_k}$, as is easy to check.
\end{proof}

Given locally convex Haudorff spaces $X, Y$ we denote by $\cL(X,Y)$ the space of all continuous linear operators from $X$ into $Y$. If $X=Y$, then we simply write $\cL(X)$ for $\cL(X,X)$. A linear map $T\colon X\to Y$ is called \textit{compact} if there exists a neighbourhood $U$ of $0$ in $X$ such that $T(U)$ is a relatively compact set in $Y$. It is clear that necessarily $T\in \cL(X,Y)$.

Given a locally convex Hausdorff space $X$ and $T\in \cL(X)$, the resolvent set $\rho(T;X)$ of $T$ consists of all $\lambda\in\C$ such that the resolvent $R(\lambda,T):=(\lambda I-T)^{-1}$ exists in $\cL(X)$. The set $\sigma(T;X):=\C\setminus \rho(T;X)$ is called the \textit{spectrum} of $T$. The \textit{point spectrum}  $\sigma_{pt}(T;X)$ of $T$ consists of all $\lambda\in\C$ (also called eigenvalues of $T$) such that $\lambda I-T$ is not injective. For the spectral theory of compact operators in locally convex spaces we refer to \cite{Ed,Gr}, for example. In particular, by \cite[Theorem 9.10.2]{Ed}, the spectrum $\sigma(T;X)$ of a compact operator $T\in \cL(X)$ is either finite or is a countable sequence converging to $0$, every non-zero $\lambda \in \sigma(T;X)$ is an eigenvalue of $T$ and the associated spectral manifold $X_{\lambda}:= \{x \in X \ | \ Tx = \lambda x \}$ is of finite dimension.

The proof of our next result requires some preliminaries from \cite{APe} and \cite{Pe} about the resolvent $R(\lambda,\sC)$, in particular for ${\rm Re}\, \lambda = 0, \ \lambda \neq 0$. In the notation of \cite{APe}, the Ces\`aro operator $\sC$ is the generalized Ces\`aro operator $\sC_g(z):=(1/z) \int_0^z f(\zeta)g'(\zeta) d \zeta$ for an analytic function $g \in H(\D)$ with $g'(z)=1/(1-z)$. In order to find the solution $f$ of the resolvent equation $\lambda f - \sC_g f = h, \ h \in H(\D)$, the analytic function $u$ in $\D \setminus (-1,0]$ is introduced in \cite{APe} (2.3)
$$
u(z)=z^{g'(0)} \exp\Big(\int_0^z \frac{g'(\zeta)-g'(0)}{\zeta} \Big), \ \ z \in \D \setminus (-1,0].
$$
This function satisfies $u(z)=z/(1-z)$ if $g'(z)=1/(1-z)$. For non positive integers $\alpha$, the powers $z^\alpha = \exp(\alpha \log_0(z))$ are defined using the principal branch of the logarithm $\log_0(z)=\log|z| + i\ {\rm arg}\ z$ with ${\rm arg}\ z \in (-\pi, \pi)$ and $z \in \D \setminus (-1,0]$.

Given an analytic function $h$ on $\D$ we denote by $m_0(h)$ the multiplicity of $0$ at the origin, and $m_0(h)=0$ if $h(0) \neq 0$. We state \cite[Proposition 2.1 (i)]{APe} since it will be used in the proof of Theorem \ref{spectrumCesaro}.

\begin{proposition}(\cite[Proposition 2.1 (i)]{APe})\label{AlemanPersson} Let $\lambda \in \C\setminus \{0\}$. If $h \in H(\D)$ satisfies $m_0(h) > {\rm Re}\frac{g'(0)
}{\lambda} - 1$, then the equation $\lambda f - \sC_g f = h$ has the analytic solution $f$ in $\D$ given by
$$
f(z)= R(\lambda,\sC)h(z)= \frac{h(z)}{\lambda} + \frac{u(z)^{1/\lambda}}{\lambda^2 z}\int_0^z u(\zeta)^{-1/\lambda} g'(\zeta) h(\zeta) d\zeta.
$$
\end{proposition}

An inspection of the proof shows that $f(z)$ is an analytic function  on $\D$ extending the function given by this formula on $\D \setminus (-1,0]$ defined using the principal branch of the logarithm.

As a consequence of Proposition \ref{AlemanPersson}, the resolvent of the Ces\`aro operator $\sC$, that is for $g'(z)=1/(1-z)$, satisfies
$$
R(\lambda,\sC) h(z) = \frac{h(z)}{\lambda} + \frac{1}{\lambda^2} z^{(1/\lambda)-1}(1-z)^{-1/\lambda} \int_0^z \zeta^{-1/\lambda} (1-\zeta)^{(1/\lambda)-1} h(\zeta) d\zeta,
$$
if $m_0(h) > {\rm Re}\frac{1}{\lambda} - 1$, since $u(z)=z/(1-z)$. This is precisely the formula $(2.7)$ in \cite{Pe}.

\begin{theorem}\label{spectrumCesaro}
The Ces\`aro operator $\sC: VH(\D) \rightarrow VH(\D)$ satisfies
\begin{itemize}
\item[(i)] $\sigma_{pt}(\sC;VH(\D))= \emptyset$.

\item[(ii)] $\sigma(\sC;VH(\D))={0} \cup \{\lambda \ | \ {\rm Re}\, \lambda > 0 \}$.
\end{itemize}
\end{theorem}
\begin{proof}
(i) By Section 2 in \cite{Pe}, the eigenvalues $\lambda$ of $\sC$ on $H(\D)$ are $\lambda= 1/n, \ n \in \N$ with associated eigenvector $e_n(z):=z^{n-1}(1-z)^{-n}$. Moreover, every eigenvector of $\lambda=1/n$ is a multiple of $e_n$. It is easy to see that $e_n \notin VH(\D)$. Consequently, $\sigma_{pt}(\sC;VH(\D))= \emptyset$.

(ii) We first show that $0 \in \sigma(\sC;VH(\D))$. The inverse of $\sC$ on $H(\D)$ is given by $\sC^{-1} f(z) = (1-z)(z f(z))'$ \cite[page 1185]{Pe}. The function $h(z):=\log(1+z), z \in \D,$ belongs to $H^\infty_{v_1} \subset VH(\D)$, but $$\sC^{-1} h(z) = (1-z) h(z) + z \frac{1-z}{1+z}$$ does not belong to $VH(\D)$. Otherwise $z \frac{1-z}{1+z} \in VH(\D)$. This is not possible, because
$$
\lim_{r \rightarrow -1} \Big|r \frac{1-r}{1+r}\Big| |\log(1-r)|^{-k} = \infty
$$
for all $k \in \N$.

The image of $((1/n) I - \sC)$ does not contain the function $z^{n-1}$, see \cite[page 1184]{Pe}. Then $1/n \in \sigma(\sC;VH(\D))$ for each $n \in \N$.

If ${\rm Re}\, \lambda >0$ and $\lambda \neq 1/n$ for each $n \in \N$, then select $\gamma>0$ such that ${\rm Re}\, (1/\lambda) \geq \gamma >0$ and $m \in \N_0$ such that $m < {\rm Re}\, (1/\lambda) < m+1$. By \cite[Proposition 4]{Pe} and its proof, there is a unique solution $e_{\lambda} \in H(\D)$ of $(\lambda \sC - I) e_{\lambda} = 1$ and it satisfies $\lim_{|z| \rightarrow 1}|e_{\lambda}(z)|(1-|z|)^{\gamma} \neq 0$, then $e_\lambda \notin A^{-\gamma}_0$. Therefore, $e_\lambda \notin VH(\D)$, and $\lambda \in \sigma(\sC;VH(\D))$.

We now fix $\lambda \in \C$ with ${\rm Re}\, \lambda < 0$. As in \cite[Section 3]{Pe}, we consider the functions
$$
\varphi_t(z):= \frac{e^{-t} z}{(e^{-t}-1)z +1}, \ \ z \in \D, \ t \geq 0,
$$
and the operators, for $f \in VH(\D)$,
$$
S_t(f)(z):= \frac{\varphi_t(z)}{z} f(\varphi_t(z)), \ \ z \in \D, \ t \geq 0.
$$
The map $\varphi_t$ is analytic on $\D$ and satisfies $\varphi_t(\D) \subset \D$ and $\varphi_t(0)=0$. By Schwarz Lemma, $|\varphi_t(z)| \leq |z|$ for all $z \in \D$. We can apply \cite[proposition 3.1]{CH} (see also \cite[Theorem 51]{B}) to conclude that $S_t: H^\infty_{v_k} \rightarrow H^\infty_{v_k}$ is continuous for each $k$ and $\|S_t\| \leq 1$.

Since ${\rm Re}\, \lambda < 0$, we have ${\rm Re}\, (1/\lambda) < 0$ and we conclude that the identity $(3.16)$ in \cite{Pe} for the resolvent of the Ces\`aro operator holds for $h \in VH(\D)$:
$$
R(\lambda,\sC) h(z) = \frac{h(z)}{\lambda} + \frac{1}{\lambda^2} \int_0^{\infty} e^{t/\lambda} S_t h(z) dt.
$$
Therefore, for each $h \in H^\infty_{v_k}$ and each $k \in \N$, we get
$$
\|R(\lambda,\sC) h\|_{v_k} \leq \frac{\|h\|_{v_k}}{|\lambda|} + \frac{\|h\|_{v_k}}{|\lambda|^2} \int_0^{\infty} e^{t {\rm Re}\, (1/\lambda)} \|S_t\| dt \leq C_{\lambda} \|h\|_{v_k},
$$
because $||S_t|| \leq 1$ and $\int_0^{\infty} e^{t {\rm Re}\,(1/\lambda)} dt< \infty$.

This implies that $R(\lambda,\sC): VH(\D) \rightarrow VH(\D)$ is continuous and $\{\lambda \ : \ {\rm Re}\, \lambda < 0 \} \subset \rho(\sC;VH(\D))$.

It remains to consider the case ${\rm Re}\, \lambda = 0, \ \lambda \neq 0$. In this case, the multiplicity $m_0(h)$ of the zero of each analytic function $h \in VH(\D)$ at the origin is greater that ${\rm Re}\, (1/\lambda) - 1 = -1$ and we can apply Proposition \ref{AlemanPersson} (\cite[Proposition 2.1 (i)]{APe}) and $(2.7)$ in \cite{Pe} to get, for $h \in H^\infty_{v_k}$,
$$
R(\lambda,\sC) h(z) = \frac{h(z)}{\lambda} + \frac{1}{\lambda^2} z^{(1/\lambda)-1}(1-z)^{-1/\lambda} \int_0^z \zeta^{-1/\lambda} (1-\zeta)^{(1/\lambda)-1} h(\zeta) d\zeta.
$$
If we write $\lambda=ib, b \in \R, b \neq 0,$ then $1/\lambda = -i/b$, and this formula becomes
$$
R(\lambda,\sC) h(z) = \frac{h(z)}{ib} - \frac{1}{b^2} z^{(-i/b)-1}(1-z)^{i/b} \int_0^z \zeta^{i/b} (1-\zeta)^{(-i/b)-1} h(\zeta) d\zeta.
$$
This is an analytic function  on $\D$ extending the function defined by this formula on $\D \setminus (-1,0]$ using the principal branch of the logarithm. We estimate the integral for $z \in \D \setminus (-1,0]$. First of all
$$
\int_0^z \zeta^{i/b} (1-\zeta)^{(-i/b)-1} h(\zeta) d\zeta = z^{(i/b)+1} \int_0^1 t^{i/b}(1-tz)^{-(i/b)-1} h(tz) dt.
$$
For each $\xi \in \D \setminus (-1,0]$, since we use the principal branch of the logarithm with ${\rm arg} z \in (-\pi, \pi)$, we get
$$
|\xi^{i/b}| = |\exp((i/b)(\log|\xi| + i\ {\rm arg}\ \xi)| = |\exp(-{\rm arg}\ \xi/b)| \leq \exp(\pi/|b|),
$$
and similarly $|\xi^{-i/b}| \leq \exp(\pi/|b|)$.
Accordingly, for $z \in \D \setminus (-1,0]$, we get the estimate
$$
\Big|\int_0^z \zeta^{i/b} (1-\zeta)^{(-i/b)-1} h(\zeta) d\zeta \Big| \leq \exp(2\pi/|b|) |z| \int_0^1 |1-tz|^{-1} |h(tz)| dt \leq
$$
$$
\leq \exp(2\pi/|b|) |z| \max_{|\zeta| = |z|} |h(\zeta)| \int_0^1 |1-tz|^{-1} dt \leq \exp(2\pi/|b|) |z| (-\log(1-|z|)) \max_{|\zeta| = |z|} |h(\zeta)|.
$$
Then, for $z \in \D \setminus (-1,0]$ with $|z| \geq 1-1/e$,
$$
v_{k+1}(z)|R(\lambda,\sC) h(z)| \leq v_{k+1}(z) \frac{|h(z)|}{|b|} + v_{k+1}(z) \frac{1}{b^2 |z|} \exp(4\pi/|b|) |z| (-\log(1-|z|)) \max_{|\zeta| = |z|} |h(\zeta)| \leq
$$
$$
\leq v_k(z) \Big(\frac{1}{|b|} + \exp(4\pi/|b|) \frac{1}{b^2} \Big) \max_{|\zeta| = |z|} |h(\zeta)|.
$$
Since $R(\lambda,\sC) h$ is an analytic function on $\D$ satisfying this estimate for $z \in \D \setminus (-1,0]$, we conclude
$$
\|R(\lambda,\sC) h\|_{v_{k+1}} \leq \Big(\frac{1}{|b|} + \exp(4\pi/|b|) \frac{1}{b^2} \Big) \|h\|_{v_k},
$$
and the resolvent operator $R(\lambda,\sC): VH(\D) \rightarrow VH(\D)$ is also continuous in this case.
\end{proof}

Observe that the methods utilized in \cite{ABR2} to determine the spectrum of $\sC$ on Korenblum type (LB)-spaces of analytic funtions cannot be used in the present setting, because $\sC$ is not continuous on the steps of the inductive limit. This is why direct arguments are necessary.

\begin{corollary}\label{noncompactCesaro}
The Ces\`aro operator $\sC: VH(\D) \rightarrow VH(\D)$ is not compact.
\end{corollary}
\begin{proof}
This is a direct consequence of Theorem \ref{spectrumCesaro} and the shape of the spectrum of compact operators, see \cite[Theorem 9.10.2]{Ed}.
\end{proof}

Let $X$ be a locally convex Hausdorff space. An operator $T\in \cL(X)$ is called \textit{power bounded} if the sequence of iterates $\{T^n : \ n\in\N_0\}$ is an equicontinuous subset of $\cL(X)$. For a Banach space $X$, this means that $\sup_{n\in\N_0}\|T^n\|<\infty$.

The space $\cL(X)$ equipped with the topology of pointwise convergence on $X$ (i.e., the strong operator topology) is denoted by $\cL_s(X)$ and when it is  equipped with the topology $\tau_b$ of uniform convergence on the bounded subsets of $X$ it is denoted by $\cL_b(X)$.

Given $T\in \cL(X)$, the averages
\[
T_{[n]}:=\frac{1}{n}\sum_{m=1}^nT^m,\quad n\in\N,
\]
are usually called the Cesàro means of $T$. The operator $T$ is said to be \textit{mean ergodic} (resp., \textit{uniformly mean ergodic}) if $(T_{[n]})_{n\in\N}$ is a convergent sequence in $\cL_s(X)$ (resp., in $\cL_b(X)$). It is easy to see that $\frac{T^n}{n}=T_{[n]}-\frac{n-1}{n}T_{[n-1]}$, for $n\geq 2$, and hence, $\tau_s$-$\lim_{n\to\infty}\frac{T^n}{n}=0$ whenever $T$ is mean ergodic. Every power bounded operator on a Montel space $X$ is necessarily uniformly mean ergodic, \cite[Proposition 5.11]{BJP}.

Concerning the linear dynamics of $T\in \cL(X)$, with $X$ a  lcHs, the operator $T\in \cL(X)$ is called \textit{hypercyclic} if there exists $x\in X$ such that the orbit $\{T^nx\colon n\in\N_0\}$
is dense in $X$. If, for some $z\in X$, the projective orbit $\{\lambda T^n z\colon \lambda\in\C,\ n\in\N_0 \}$ is dense
in $X$, then $T$ is called \textit{supercyclic}. Clearly, hypercyclicity  implies supercyclicity.

More details for mean ergodic operators  can be seen in \cite{BJP,DSI,K}, and for linear dynamics in \cite{B-M,G-P}.

\begin{proposition}\label{meanergodicCesaro}
The Ces\`aro operator $\sC: VH(\D) \rightarrow VH(\D)$ is not power bounded, not mean ergodic and not supercyclic.
\end{proposition}
\begin{proof}
By Yosida's Theorem \cite[Theorem 5.5]{BJP}, if $\sC: VH(\D) \rightarrow VH(\D)$ is mean ergodic, then
$$
VH(\D) = {\rm ker}(I-\sC) \oplus \overline{{\rm Im}(I-\sC)}.
$$
Theorem \ref{spectrumCesaro} (i) implies that ${\rm ker}(I-\sC) = \{ 0 \}$. On the other hand $\overline{{\rm Im}(I-\sC)} \subset \{ f \in VH(\D) \ | \ f(0)=0 \}$, which is a proper closed hyperplane of $VH(\D)$. This implies that $\sC: VH(\D) \rightarrow VH(\D)$ is not mean ergodic.

On the other hand, if $\sC: VH(\D) \rightarrow VH(\D)$ were power bounded, then it would be uniformly mean ergodic, since $VH(\D)$ is a Montel space. We have just seen that this operator is not mean ergodic.

The Ces\`aro operator $\sC: H(\D) \rightarrow H(\D)$ is not supercyclic by \cite[proposition 3.5]{ABR_power} and the space $VH(\D)$ is dense in $H(\D)$ since it contains the polynomials. The comparison principle \cite[Subsection 1.1.1]{B-M} ensures that $\sC: VH(\D) \rightarrow VH(\D)$ is not supercyclic.
\end{proof}


\section{Generalized Ces\`aro operators.}\label{genCesoper}

For $t\in [0,1]$ the generalized Ces\`aro operator   $C_t\colon H(\D)\to H(\D)$,  for $f \in H(\D)$,  is defined by
$$ C_tf(0):=f(0)$$
and
\begin{equation}\label{eq.formula-int}
C_tf(z):=\frac{1}{z}\int_0^z\frac{f(\zeta)}{1-t\zeta}\,d\zeta,\ z\in \D\setminus\{0\},
\end{equation}

The operator $C_t$ also has the following representation

\begin{align}
	&C_tf(z)=\sum_{n=0}^\infty\left(\frac{t^na_0+t^{n-1}a_1+\ldots +a_n}{n+1}\right)z^{n}\nonumber\\
	&=\sum_{n=0}^\infty\left(\frac{t^n\hat{f}(0)+t^{n-1}\hat{f}(1)+\ldots +\hat{f}(n)}{n+1}\right)z^{n},
\end{align}

\noindent
where the coefficients of the series are precisely those of the discrete generalized Ces\`aro operator, which is defined for  $x \in\mathbb{C}^{\mathbb{N}_0}$ by

$$
	C_tx:=\left(\frac{t^nx_0+t^{n-1}x_1+\ldots +x_n}{n+1}\right)_{n},\quad x=(x_n)_{n}\in\mathbb{C}^{\mathbb{N}_0}.
$$

The operator $C_0$ is given by $C_0f(z)=\frac{1}{z}\int_0^z f(\zeta)\,d\zeta$ for $z\not=0$ and $C_0f(0)=f(0)$, which is the classical Hardy operator in $H(\D)$, and the operator $C_1$ is precisely the Ces\`aro operator $\sC$ investigated in Section \ref{Cesoper}.

The (discrete) generalized Cesàro operators $C_t$, for $t\in [0,1]$, were first investigated by Rhaly, \cite{R1}, \cite{R2}. We refer the reader to the introduction of the paper \cite{CR4} and the references therein. The behaviour of $C_t, 0 \leq t < 1$ on $H(\D)$ and on the weighted spaces $H_v^\infty$ and $H_v^0$ was investigated recently in \cite{ABR4}. The following results are proved in this article.

\begin{proposition}\label{from ABRrecent}
\begin{itemize}
\item[(i)] (\cite[Propositions 2.1 and 3.3.]{ABR4}) For every $t\in [0,1]$ the  operator $C_t\colon H(\D)\to H(\D)$ is a surjective continuous isomorphism.

\item[(ii)] (\cite[Propositions 2.5 and 2.7]{ABR4}) For each  weight function $v$ and for each $t\in [0,1)$ the  operators $C_t\colon H_v^\infty\to H_v^\infty$ and $C_t\colon H^0_v\to H^0_v$ are continuous and even compact.
\end{itemize}
\end{proposition}

Proposition \ref{from ABRrecent} (ii) and Lemma \ref{continuityLB} imply

\begin{proposition}\label{continuityCt}
For every $t\in [0,1)$ the  operator $C_t\colon VH(\D)\to VH(\D)$ is continuous.
\end{proposition}

The following abstract result is well known. We briefly indicate the proof for the sake of completeness.

\begin{lemma}\label{compact-LB}
Let $X=\ind_k X_k$ be a Hausdorff inductive limit of a sequence of Banach spaces $(X_k,\|\cdot\|_k)$ such that the inclusion $X_k \subset X_{k+1}$ is compact. An operator $T\in \cL(X)$ is compact if and only if there is $m$ such that for all $k \in \N$ we have $T(X_k) \subset X_m$.
	\end{lemma}
\begin{proof}
Assume first that $T$ is compact. There is a $0$-neighbourhood $U$ in $X$ such that $T(U)$ is relatively compact in $X$. Since all the inclusions $X_k \subset X_{k+1}$ are compact, the (LB)-space $X$ is a (DFS)-space, hence there is $m \in \N$ such that $T(U)$ is contained and relatively compact in $X_m$. This implies $T(X) \subset X_m$, and $T(X_k) \subset X_m$ for all $k \in \N$.

Now we suppose that $T\in \cL(X)$ satisfies $T(X_k) \subset X_m$ for all $k \in \N$. The continuity of $T: X \rightarrow X$ and of the inclusions $X_k \subset X$ for all $k$ imply that $T:X_k \rightarrow X_m$ has closed graph. By the closed graph theorem, $T:X_k \rightarrow X_m$ is continuous for each $k$. Let us denote by $B_k$ the closed unit ball of $X_k$. For each $k$ there is $\mu_k > 0$ such that $T(\mu_k B_k) \subset B_m$. Then, the absolutely convex hull $U$ of $ \bigcup_{k} \mu_k B_k$ is a $0$-neighbourhood in $X$ such that $T(U) \subset B_m$. Since the inclusion $X_m \subset X_{m+1}$ is compact, $T(U)$ is relatively compact in $X_{m+1}$, consequently in $X$, and $T:X \rightarrow X$ is indeed a compact operator.
\end{proof}

\begin{proposition}\label{compactCt}
For every $t\in [0,1)$ the  operator $C_t: VH(\D)\to VH(\D)$ is compact.
\end{proposition}
\begin{proof}
We consider the standard weight $w_1(r):=1-r, \ 0 \leq r < 1$ and the weight $v_1(z)= 1$ if $|z| \leq 1-1/e$, and $v_1(z)= (-\log(1-r))^{-1}$ if $|z| \geq 1-1/e$ and we show that $C_t: H_{w_1}^\infty\to H_{v_1}^\infty$ is continuous. Indeed, given $f \in H_{w_1}^\infty, \ ||f||_{w_1}\leq 1 $, we have for $|z| \geq 1-1/e$
$$
|C_tf(z)| = \Big|\frac{1}{z}\int_0^z\frac{f(\zeta)}{1-t\zeta}\,d\zeta \Big| = \Big|\int_0^1\frac{f(sz)}{1-tsz}\,ds \Big| \leq \int_0^1\frac{|f(sz)|}{|1-tsz|}\,ds.
$$
Since $|f(sz)| \leq 1/(1-s|z|)$ and $|1-tsz| \geq 1-t$, we can continue to estimate as follows
$$
|C_tf(z)| \leq \frac{1}{1-t} \int_0^1\frac{1}{1-s|z|}\,ds = \frac{1}{1-t} \Big(-\frac{\log(1-|z|)}{|z|}\Big) \leq \frac{1}{(1-t)(1-1/e)} \frac{1}{v_1(z)}.
$$
Consequently,
$$
\sup_{|z| \geq 1-1/e} v_1(z) |C_tf(z)| \leq \frac{1}{(1-t)(1-1/e)}
$$
for each $f \in H_{w_1}^\infty, \ ||f||_{w_1}\leq 1 $, and this implies the continuity of $C_t: H_{w_1}^\infty\to H_{v_1}^\infty$.

Now, each Banach space $H_{v_k}^\infty$ is continuously included in $H_{w_1}^\infty$, hence in $H_{v_1}^\infty$. The compactness of
$C_t: VH(\D)\to VH(\D)$ now follows from Proposition \ref{continuityCt} and Lemma \ref{compact-LB}.
\end{proof}

\begin{proposition}\label{SpectrumCt}  For each $t\in [0,1)$ the operator $C_t: VH(\D)\to VH(\D)$ satisfies:
\begin{itemize}
\item[(i)] $\sigma_{pt}(C_t, VH(\D))=\left\{\frac{1}{m+1}\,: m\in\N_0\right\}$.

\item[(ii)] $\sigma(C_t, VH(\D))=\left\{\frac{1}{m+1}\,: m\in\N_0\right\}\cup\{0\}$.
\end{itemize}
\end{proposition}

\begin{proof} Fix $t\in [0,1)$.

(i) The operator $C_t: H(\D)\to H(\D)$ is continuous and its point spectrum satisfies $\sigma_{pt}(C_t, H(\D))=\left\{\frac{1}{m+1}\,: m\in\N_0\right\}$ by \cite[Proposition 3.7]{ABR4}. Since $VH(\D) \subset H(\D)$ with continuous inclusion, it follows that
$\sigma_{pt}(C_t, VH(\D)) \subset \left\{\frac{1}{m+1}\,: n\in\N_0\right\}$.

On the other hand, for each $m\in\N_0$,  the corresponding eigenspace ${\rm ker}(\frac{1}{m+1}I-C_t)$ is 1-dimensional and is generated by an eigenvector $g_m(z):=\sum_{n=0}^\infty (x^{[m]})_nz^n$ with $(x_n^{[m]})_{n\in\N_0}\in \ell^1$. See the proofs of \cite[Propositions 2.8 and 3.7]{ABR4}. Then $g_m(z) \in H^\infty(\D) \subset VH(\D)$ and $\frac{1}{m+1} \in \sigma_{pt}(C_t, VH(\D))$ for each $m\in\N_0$.

(ii) This follows from part (i) and the fact that $C_t: VH(\D)\to VH(\D)$ is a compact operator. See \cite[Theorem 9.10.2]{Ed}.
\end{proof}

The next result is a particular case of \cite[Theorem 6.4]{ABR3}.

\begin{lemma}\label{powerbddLB}
Let $X$ be a complete (LB)-space. Let $T\in \cL(X)$ be a compact operator such that $1\in\sigma(T;X)$, $\sigma(T;X)\setminus\{1\}\su \ov{B(0,\delta)}$ for some $\delta\in (0,1)$  and satisfying ${\rm ker} (I-T)\cap {\rm Im}(I-T)(X)=\{0\}$. Then $T$ is power bounded and uniformly mean ergodic.
\end{lemma}

\begin{proposition}\label{meanergodicCt}
For every $t\in [0,1)$ the  operator $C_t: VH(\D)\rightarrow VH(\D)$ is  power bounded, uniformly mean ergodic, but it is not supercyclic.
\end{proposition}
\begin{proof}
We check that all the assumptions of Lemma \ref{powerbddLB} are satisfied to conclude that $C_t$ is power bounded and uniformly mean ergodic on $VH(\D)$.

It is easy to see that ${\rm ker} (I-C_t)={\rm span}\{g_0\}$, with $g_0(z)=\sum_{n=0}^\infty t^nz^n \in VH(\D)$, for $z\in\D$, see \cite{ABR4}. On the other hand, ${\rm Im}(I-C_t)$ is a closed subspace of $VH(\D)$ since $C_t$ is compact on $VH(\D)$ by Proposition \ref{compactCt},  and  ${\rm Im}(I-C_t)\subset \{g\in VH(\D)\,:\, g(0)=0\}$. Moreover,  \cite[Theorem 9.10.1]{Ed} implies that ${\rm codim}\,{\rm Im}(I-C_t)={\rm dim}\,{\rm ker} (I-C_t)=1$. Accordingly, both ${\rm Im}(I-C_t)$ and $\{g\in H^\infty_v\, :\, g(0)=0\}$ are hyperplanes of $VH(\D)$. Hence ${\rm Im}(I-C_t)=\{g\in H^\infty_v\, :\, g(0)=0\}$.

Now, if $h\in {\rm Im}(I-C_t)\cap{\rm ker}(I-C_t)$, then $h(0)=0$ and there exists $\lambda\in \C$ such that $h=\lambda g_0$. This yields that $0=h(0)=\lambda g_0(0)=\lambda$. Hence,  ${\rm Im}(I-C_t)\cap {\rm ker} (I-C_t)=\{0\}$.

Finally, Proposition \ref{SpectrumCt} implies that $1\in \sigma(C_t,VH(\D))=\{\frac{1}{m+1}\,;\, m\in\N_0\}\cup\{0\}$. Consequently, $\sigma(C_t,VH(\D))\setminus\{1\}\subset \overline{B(0,1/2)}$.

The  operator $C_t: H(\D)\rightarrow H(\D)$ is not supercyclic for each $t\in [0,1)$ by \cite[Theorem 3.8]{ABR4}. The comparison principle \cite[Subsection 1.1.1]{B-M} ensures that $C_t: VH(\D)\rightarrow VH(\D)$ is not supercyclic.
\end{proof}

\textbf{Acknowledgement.}
We thank the referee for the careful reading of the manuscript and for the suggestion to improve and clarify  the proof of Theorem \ref{spectrumCesaro}.

This research was partially supported by the Project PID2020-119457GB-100
funded by MCIN/AEI/10.13039/501100011033 and by ``ERDF A way of making Europe''.


\bigskip
\bibliographystyle{plain}

\begin{thebibliography}{99}

\bibitem{ABR1} A.A. Albanese, J. Bonet, W.J. Ricker, The Ces\`{a}ro operator in growth Banach spaces of analytic functions. Integral Equations Operator Theory 86 (2016), no. 1, 97--112.

\bibitem{ABR2} A.A. Albanese, J. Bonet, W.J. Ricker, The Ces\`{a}ro operator on Korenblum type spaces of analytic functions. Collect. Math. 69 (2018), no. 2, 263--281.

\bibitem{ABR_power} A.A. Albanese, J. Bonet, W.J. Ricker, The Ces\`aro operator on power series spaces, Studia Math. 240 (2018) 47--68.

\bibitem{ABR3} A.A. Albanese, J. Bonet, W.J. Ricker, Spectral properties of generalized Cesáro operator in sequence spaces. Rev. Real Acad. Cienc. Exactas Fis. Nat. Ser. A-Mat. 117, Article number 140 (2023). https://doi.org/10.48550/arXiv.2305.04805

\bibitem{ABR4} A.A. Albanese, J. Bonet, W.J. Ricker, Generalized Cesàro operators in weighted Banach spaces of analytic functions with sup-norms, Collect. Math. (2024). https://doi.org/10.1007/s13348-024-00437-9

\bibitem{APe} A. Aleman, A.-M. Persson, Resolvent estimates and decomposable extensions of generalized Ces\`aro operators, J. Funct. Anal. 258 (2010), 67--98.

\bibitem{B-M} F. Bayart, E. Matheron, Dynamics of linear operators, Cambridge Tracts in Mathematics, vol. 179, Cambridge University Press, Cambridge, 2009.

\bibitem{Bi} K.D.\ Bierstedt, An introduction to locally convex inductive limits,  Functional analysis and its applications (Nice, 1986), 35--133,
ICPAM Lecture Notes, World Sci. Publishing, Singapore, 1988.

\bibitem{BBG}  K.D.\ Bierstedt, J.\ Bonet, A.\ Galbis, Weighted spaces of
holomorphic functions on bounded domains, Michigan Math.\
J. 40 (1993) 271--297.

\bibitem{BBT} K.D.\ Bierstedt, J.\ Bonet, J.\ Taskinen, Associated weights and spaces of holomorphic functions, Studia Math.  127
(1998), 137--168.

\bibitem{BMS} K.D. Bierstedt, R. Meise, W.H. Summers, A projective description of weighted inductive limits, Trans. Amer. Math. Soc. 272 (1982) 107--160.

\bibitem{B} J. Bonet, Weighted Banach spaces of analytic functions with sup-norms and operators between them: a survey. Rev. R. Acad. Cienc. Exactas Fís. Nat. Ser. A Mat. 116(4), Article number 184 (2022)

\bibitem{BJP} J. Bonet, D. Jornet, P. Sevilla-Peris, Function Spaces and Operators between them, RSME Springer Series 11, Springer 2023.

\bibitem{BT} J. Bonet, J. Taskinen, Toeplitz operators on the space of analytic functions with logarithmic growth, J. Math. Anal. Appl. 353 (2009) 428--435

\bibitem{CH} M.D. Contreras, A.G. Hern\'{a}ndez-D\'{\i}az,
Weighted composition operators in weighted Banach spaces of
analytic functions, J. Austral. Math. Soc.  69 (2000), 41--60.

\bibitem{CR4} G.P. Curbera, W.J. Ricker, Fine spectra and compactness of generalized Cesàro operators in Banach lattices in $\C^{\N_0}$, J. Math. Anal. Appl. \textbf{507} (2022), Article Number 125824 (31 pp.).

\bibitem{DSI} N. Dunford, J.T. Schwartz, Linear Operators I: General Theory. 2nd Printing, Wiley Interscience Publ., New York, 1964.

\bibitem{Ed} R.E. Edwards, Functional Analysis, Theory and Applications,  Holt, Rinehart and Winston, New York-Chicago-San Francisco, 1965.

\bibitem{GGP} D. Girela, C. Gonz\'alez, J.A. Pel\'aez, Multiplication and division by inner functions in the space of Bloch functions, Proc. Amer. Math. Soc. \textbf{134} (2006), 1309--1314.

\bibitem{G-P} K.-G. Grosse-Erdmann, A. Peris Manguillot, Linear Chaos, Universitext, Springer Verlag, London, 2011.

\bibitem{Gr} A. Grothendieck, Topological Vector Spaces, Gordon and Breach, London, 1973.

\bibitem{HKZ} H. Hedenmalm, B. Korenblum, K. Zhu, Theory of Bergman Spaces. Grad. Texts in Math. 199, Springer-Verlag, New York, 2000.

\bibitem{J} M. Jasiczak, On locally convex extension of $H^\infty$ in the unit ball and continuity of the Bergman projection, Studia Math. 156 (2003) 261--275.

\bibitem{K} U. Krengel,  Ergodic Theorems, Walter de Gruyter, Berlin, 1985.

\bibitem{L2} W. Lusky,  On the isomorphism classes of weighted spaces of harmonic and holomorphic functions. Studia Math. 175 (2006), 19--40.

\bibitem{MV} R. Meise, D. Vogt, Introduction to Functional Analysis, Clarendon, Oxford, 1997.

\bibitem{Pe} A.-M. Persson, On the spectrum of the Ces\`aro operator on spaces of analytic functions,  J. Math. Anal Appl. \textbf{340} (2008), 1180--1203.

\bibitem{R1} H.C. Jr. Rhaly, Discrete generalized Cesàro operators, Proc. Amer. Math. Soc. \textbf{86} (1982), 405-409.

\bibitem{R2} H.C. Jr. Rhaly, Generalized Cesàro matrices, Canad. Math. Bull. \textbf{27} (1984), 417-422.

\bibitem{T} J. Taskinen, On the continuity of the Bergman and Szegö projections, Houston J. Math. 30 (2004) 171--190.

\bibitem{Z} K. Zhu, Operator Theory on Function Spaces, Math. Surveys and Monographs Vo. 138. Amer. Math. Soc. 2007.



\end{thebibliography}

%
%
%


\noindent \textbf{Author's address:}%
\vspace{\baselineskip}%

Jos\'e Bonet: Instituto Universitario de Matem\'{a}tica Pura y Aplicada IUMPA,
Universitat Polit\`{e}cnica de Val\`{e}ncia,  E-46071 Valencia, Spain

email: jbonet@mat.upv.es \\

\end{document}